\documentclass{amsart}
\usepackage{amssymb,hyperref} 
\newtheorem{thm}{Theorem}[section]

\newtheorem{lem}[thm]{Lemma}
\newtheorem{prop}[thm]{Proposition}
\theoremstyle{definition}

\newtheorem{rem}[thm]{Remark}

\newtheorem{qu}[thm]{Question}
\newtheorem{con}[thm]{Conjecture}
\numberwithin{equation}{section}

\newcommand{\set}[1]{\left\{#1\right\}}

\newcommand{\pr}[1]{\left(#1\right)}
\newcommand{\valu}[1]{\left\langle #1\right\rangle}
\newcommand{\ra}[2]{#1\rightarrow #2}
\newcommand{\ma}[2]{#1\mapsto #2}
\begin{document}

\title[Profinite groups with many elements of bounded order]{Profinite groups with many elements of bounded order}%

\author[A. Abdollahi]{Alireza Abdollahi}%
\address{Department of Pure Mathematics, Faculty of Mathematics and Statistics, University of Isfahan, Isfahan 81746-73441, Iran.} 
\email{a.abdollahi@math.ui.ac.ir}%
\author[M. Soleimani Malekan]{Meisam Soleimani Malekan}%
\address{Department of Pure Mathematics, Faculty of Mathematics and Statistics, University of Isfahan, Isfahan 81746-73441, Iran; 
	Institute for Research in Fundamental Sciences, School of Mathematics, Tehran, Iran.} 
\email{msmalekan@gmail.com}


\begin{abstract}
L\'evai  and  Pyber proposed the following as a conjecture:\\  
 Let $G$ be a profinite group such that the set of solutions of the equation $x^n=1$
has positive Haar measure. Then $G$ has an open subgroup $H$ and an element $t$ such that all
elements of the coset $tH$ have order dividing $n$ (see Problem 14.53 of [The Kourovka Notebook, No. 19, 2019]). \\
We define a constant $c_n$ for all finite groups and prove that the latter conjecture is equivalent with a conjecture saying $c_n<1$. Using the latter equivalence  we observe that correctness of L\'evai and Pyber conjecture  implies the existence of the universal upper bound $\frac{1}{1-c_n}$ on the index of generalized Hughes-Thompson subgroup $H_n$ of finite groups whenever it is non-trivial. It is known that the latter is widely open even for all primes $n=p\geq 5$.  For odd $n$ we also prove that L\'evai and Pyber conjecture is equivalent to show that  $c_n$ is less than $1$ whenever $c_n$ is only computed on finite solvable groups. \\
The validity of the conjecture has been proved in [Arch. Math. (Basel) 75 (2000) 1-7] for $n=2$. Here we confirm the conjecture for $n=3$. 
\end{abstract}
\maketitle
\section{\bf Introduction and Results}

Let $G$ be a Hausdorff compact group. Then $G$ has a unique normalized Haar measure denoted by ${\mathbf m}_G$. In general, the question of weather the interior of every non-empty measurable subset of $G$ with positive Haar measure is non-empty has negative answer even if $G$ is profinite (see e.g. \cite{LP}). However the same question for subsets defined by words is still open. In \cite{LP} the following conjecture is proposed.   

\begin{con}\label{l-pcon}{\rm (Conjecture 3 of \cite{LP}, Problem 14.53 of \cite{Ko})}   
Let $G$ be a profinite group such that the set $X_n(G)$ of solutions of the equation $x^n=1$ in $G$  has positive Haar measure. Then $G$ has an open subgroup $H$ and an element $t$ such that all elements of the coset $tH$ have order dividing $n$.
\end{con} 

The validity of Conjecture \ref{l-pcon} has been proved in \cite{LP} for $n=2$. In \cite{SAE} it is shown that the conjecture is valid for $n=2$ even if $G$ is   Hausdorff compact.  It is also proved in \cite{SAE} that if $X_3(G)$ has positive Haar measure in a compact group $G$, then $G$ contains an open normal subgroup which is $2$-Engel. Here we confirm Conjecture \ref{l-pcon} for $n=3$.
To do so, we first show that Conjecture \ref{l-pcon} is equivalent to the following one. We need the following notation in the statement of the conjecture. For an arbitrary group $K$ and an automorphism $\phi$ of $K$ of order dividing a positive integer $n$, define 
$$X_{n,\phi}(K):=\left\{x\in K \;|\; xx^\phi x^{\phi^2}\cdots x^{\phi^{n-1}}=1\right\}.$$ 
 The automorphism group of $K$ will be denoted by ${\rm Aut}(K)$.

\begin{con}\label{sup<1}
	  $$ c_n:=\sup \pr{\left\{\frac{|X_{n,\phi}(H)|}{|H|}: H\,
	\text{\rm is a finite group and}\;\; \phi\in{\rm Aut}(H), \; \phi^n={\rm id}\right\} \setminus\set{1}}<1.$$
\end{con}

 It is known that Conjecture \ref{sup<1} is valid for $n=2$ and the supremum $c_2$ is $\frac{3}{4}$ (see \cite{Ma}). We shall prove that $c_3<1$. 

If $n$ is odd, using \cite[Theorem 1.10]{LS}, we prove that Conjecture \ref{l-pcon} is equivalent to the following. Here we denote by $\mathcal{S}$ the class of finite solvable groups. 
\begin{con}\label{supS<1}
	$$ c^{\mathcal{S}}_{2n+1}:=\sup \pr{\left\{\frac{|X_{n,\phi}(H)|}{|H|}: H\in \mathcal{S}\,
		\text{\rm and}\;\; \phi\in{\rm Aut}(H), \; \phi^{2n+1}={\rm id}\right\} \setminus\set{1}}<1.$$
\end{con}

We conclude this section by noting  how above conjectures are difficult. First we note that  
\begin{prop}
	If Conjecture \ref{sup<1} is true for $n$ i.e. $c_n<1$, then for any finite group $G$ with $H_n(G):=\langle x\in G \;|\; x^n\neq 1\rangle\neq 1$, $|G:H_n(G)|<\frac{1}{1-c_n}$.
\end{prop}
\begin{proof}
	Since $H_n(G)\neq 1$, $X_n(G)\neq G$. Now the proof follows from the fact that $G\setminus H_n(G) \subset X_n(G)$ and the definition of $c_n$. 
\end{proof}

\begin{rem}
	The problem of finding a universal bound from above depending only on $p$ for  the index $|G:H_p(G)|$  of the Hughes-Thompson subgroup $H_p(G)$ of finite $p$-groups $G$ with $H_p(G)\neq 1$ is widely open for all $p\geq 5$ (see \cite[Chapter 7]{Kh} for the history and results on the problem).
\end{rem}

We finish this section with the following two questions.

\begin{qu}
	Suppose that $n$ is a positive integer such that there exists a positive integer $k_n$ depending only on $n$ such that $|G:H_n(G)|\leq k_n$ for all finite groups $G$ with $1\neq H_n(G)$. Is it true that $c_n<1$? The same question whenever $n$ is prime.
\end{qu}

\begin{qu}
	Let $n>1$ be a positive integer such that $c_d<1$ for all prime power divisors $d$ of $n$. Is it true that $c_n<1$?
\end{qu}

\section{\bf Profinite groups}

We denote the normalized Haar measure of a compact group $G$ by ${\mathbf m}_G$, and we will simply write ${\mathbf m}$ if there is no ambiguity. \\

The following easy lemma will be used in the sequel without referring to it. 
 
\begin{lem}{\rm (cf. \cite[Lemma 2.5]{jw2020})}\label{n-large}
	Let $G$ be a compact group and $A\subseteq G$ be a measurable subset. Assume that ${\mathbf m}(A)\geq 1-\epsilon$, then ${\mathbf m}\pr{\bigcap_{k=1}^ng_kA}\geq 1-n\epsilon$ for all $g_1,\dotsc,g_n\in G$. The similar result with  strict inequalities holds. 
\end{lem}
\begin{proof}
	By induction on $n$, we prove the result. For $n=1$, it holds as the measure is left-invariant. Assume that the result is true for $n$; therefore 
	\begin{align*}
		{\mathbf m}\pr{\bigcap_{k=1}^{n+1}g_kA}&={\mathbf m}\pr{\bigcap_{k=1}^{n}g_kA}+{\mathbf m}(g_{n+1}A)-{\mathbf m}\pr{\pr{\bigcap_{k=1}^{n}g_kA}\cup(g_{n+1}A)}\\
		&\geq (1-n\epsilon)+(1-\epsilon)-1,\quad\text{by the induction hypothesis,}\\
		&=1-(n+1)\epsilon.
	\end{align*}
\end{proof}

\begin{lem}
	\label{relation} Let $G$ be a compact group and $\phi$ be a continuous automorphism of $G$ of order dividing $n$. Denote by $G\rtimes\valu{\phi}$ the semidirect product of $G$ by $\valu{\phi}$. Then: 
	\begin{itemize}
		\item[(i)] $X_{n,\phi}(G)$ has nonempty interior if and only if $X_n(G\rtimes\valu{\phi})$ has nonempty interior. 
		\item[(ii)] If $X_{n,\phi}(G)$ has positive Haar measure then $X_n(G\rtimes\valu{\phi})$ has positive Haar measure. 
	\end{itemize}
\end{lem}
\begin{proof}
	It follows from the equality $X_n(G\rtimes\valu{\phi})\cap G\phi^{-1}=X_{n,\phi}(G)\phi^{-1}$.
\end{proof}
\begin{prop}\label{cj implies cj}
 Conjecture \ref{l-pcon} implies  Conjecture \ref{sup<1}.
\end{prop}
\begin{proof}
	If $n$ is such that  Conjecture \ref{sup<1} is not valid, then there   exist sequences $(G_k)$ of finite groups and $(\phi_k)\in\prod_{k=1}^\infty\text{Aut}(G_k)$ such that  $\phi_k^n=1$  and 
	\[ 0<\prod_{k=1}^\infty\frac{|X_{n,\phi_k}(G_k)|}{|G_k|}<1.\]
	Consider the cartezian product $G=\prod_kG_k$ which is clrealy profinite. Then $\phi=(\phi_k)$ is an automorphism of  $G$ of order dividing $n$. It is clear that the measure of $X_{n,\phi}(G)$ is equal to $\prod_k\frac{|X_{n,\phi_k}(G_k)|}{|G_k|}$ and its interior is empty, so by Lemma \ref{relation}, $X_n(G\rtimes\valu{\phi})$ has positive Haar measur and empty interior showing that Conjecture \ref{l-pcon} is not valid. 
\end{proof}

The following lemma will be used in the proof that  ``Conjecture \ref{sup<1} implies  Conjecture \ref{l-pcon}". 
We write ``$N\trianglelefteq_o G$" whenever $N$ is a normal and open subgroup of $G$.
\begin{lem}\label{sup}
	Let $A$ be a closed subset of a profinite group with positive Haar measure and $M$ be any normal open subgroup of $G$. If $\mathcal{X}$ is the set of all normal open subgroups of $G$ contained in $M$. Then 
	\begin{align*}
		\sup\left\{\frac{{\mathbf m}(Ng\cap A)}{{\mathbf m}(N)}: g\in G, N\in {\mathcal X} \right\}=1
	\end{align*}
\end{lem}
\begin{proof} 
	Let $N\in {\mathcal X}$ be such that $r:=|G:N|$. If $s$ is the number of cosets of $N$ which intersect $A$, then 
	\begin{align}
		\label{ine1}
		(r-s) {\mathbf m}(N)\leq 1-{\mathbf m}(A).
	\end{align}
	On the other hand, assume that ${\mathbf m}(Nx\cap A)=\max\{{\mathbf m}(Ng\cap A): g\in G\}$, so 
	\begin{align}\label{ine2}
		{\mathbf m}(A)\leq s{\mathbf m}(Nx\cap A)
	\end{align}
It follows from inequalities (\ref{ine1}) and (\ref{ine2}) that  
	\begin{align*}
		\frac{{\mathbf m}(A)}{1-{\mathbf m}(A)}\frac{r-s}{s}\leq\frac{{\mathbf m}(Nx\cap A)}{{\mathbf m}(N)}.
	\end{align*}
	Since $A$ is closed, ${\mathbf m}(A)=\inf\left\{\frac{|AK/K|}{|G/K|}: K\trianglelefteq_o G \right\}$. Now since $\frac{|AK/K|}{|G/K|}\leq \frac{|AN/N|}{|G/N|}$, whenever $K\leq N$ are normal subgroups of $G$ of finite index, it follows that  
		 ${\mathbf m}(A)=\inf\left\{\frac{|AN/N|}{|G/N|}: N\in {\mathcal X} \right\}$. Since $\frac{r}{s}=\frac{|G/N|}{|AN/N|}$, the result now follows from the last inequality. 
\end{proof}
For any positive integer $n$ and any class  $\mathcal{X}$ of finite groups, we denote by $c^{\mathcal{X}}_n$ the following positive real number which is at most $1$:
$$ \sup \pr{\left\{\frac{|X_{n,\phi}(H)|}{|H|}: H\in \mathcal{X} \,
	 \textbf{\rm and}\;\; \phi\in{\rm Aut}(H), \; \phi^n={\rm id}\right\} \setminus\set{1}}.$$
\begin{prop}\label{**}
Assume $c^{\mathcal{X}}_n<1$ and suppose that	$G$ is a profinite group. Let $M$ be a normal open subgroup of $G$ and $\phi$ be a continues automorphism of $M$ such that $N^\phi\subset N$ for all normal open subgroups $N$ of $G$ contained in $M$ and $M/N \in {\mathcal X}$. Then 
	${\mathbf m}_M(X_{n,\phi}(M))\leq c^{\mathcal{X}}_n$ if $X_{n,\phi}(M)\neq M$.
\end{prop}
\begin{proof}
	Seeking a contradiction, let us suppose ${\mathbf m}_M(X_{n,\phi}(M))> c^{\mathcal{X}}_n$. Let $N$ be a normal open subgroup of $G$. Consider the following  automorphism of $M/(M\cap N)$,
	\[\bar\phi:\frac M{M\cap N}\rightarrow\frac M{M\cap N}, \quad \bar x^{\bar\phi}:=\overline{x^\phi}\]
	We have 
	\[{\mathbf m}_M\pr{X_{n,\phi}(M)}\leq \frac{|X_{n,\bar\phi}(M/(M\cap N))|}{|M/(M\cap N)|}\]
	the inequality holds because if $x\in X_{n,\phi}(M)$, then $\bar x\in X_{n,\bar\phi}(M/(M\cap N))$. Since $c^{\mathcal{X}}_n<1$, so $X_{n,\bar\phi}(M/(M\cap N))=M/(M\cap N)$, whence $\prod_{k=0}^{n-1}x^{\phi^k}\in N$ for all $x\in M$. Since $\bigcap\{N\trianglelefteq_o G: N\leq M\}=\{1\}$,  $\prod_{k=0}^{n-1}x^{\phi^k}=1$ for all $x\in M$, i.e. $X_{n,\phi}(M)=M$.
\end{proof}

We are now in a position to prove the following:
\begin{prop}\label{cj <- cj}
Conjecture \ref{sup<1} implies Conjecture \ref{l-pcon}.
\end{prop}
\begin{proof}
	Let $G$ be a profinite group such that $\textbf{m}_G(X_n(G))>0$. By Lemma \ref{sup}, there exist a normal open subgroup $M$ and $g\in X_n(G)$, such that 
	\begin{align}\label{*}
		c_n<\frac{{\mathbf m}_G\pr{Mg \cap X_n(G)}}{{\mathbf m}_G(M)}=\frac{{\mathbf m}_G\pr{M\cap X_n(G)g^{-1}}}{{\mathbf m}_G(M)}.
	\end{align} 
	Put $\phi:\ra MM$, $\ma x {gxg^{-1}}$. Since $g^n=1$, $\phi^n=1$ and  
	\[ M\cap X_n(G)g=X_{n,\phi}(M)\]
	and the inequality \ref{*} means that ${\mathbf m}_M(X_{n,\phi}(M))>c_n$, so by Proposition \ref{**}, $X_{n,\phi}(M)=M$, whence $Mg\subseteq X_n(G)$. 
\end{proof}

\begin{prop} \label{S->lp}
For odd $n$, Conjecture \ref{supS<1} implies Conjecture \ref{l-pcon}.
\end{prop}
\begin{proof}
Let $G$ be a profinite group such that ${\mathbf m}_G(X_n(G))>0$. By \cite[Theorem 1.10]{LS} and Lemma \ref{sup}, there exist a normal open subgroup $M$ and $g\in X_n(G)$ such that 
such that $M/N$ is finite solvable for all open normal subgroups $N$ of $G$ contained in $M$, and
	\begin{align}\label{*S}
	c^{\mathcal{S}}_n<\frac{{\mathbf m}_G\pr{Mg \cap X_n(G)}}{{\mathbf m}_G(M)}=\frac{{\mathbf m}_G\pr{M\cap X_n(G)g^{-1}}}{{\mathbf m}_G(M)}.
	\end{align} 
	Put $\phi:\ra MM$, $\ma x {gxg^{-1}}$. Since $g^n=1$, $\phi^n=1$ and  
	\[ M\cap X_n(G)g=X_{n,\phi}(M)\]
	and the inequality \ref{*S} means that ${\mathbf m}_M(X_{n,\phi}(M))>c^{\mathcal{S}}_n$, so by Proposition \ref{**}, $X_{n,\phi}(M)=M$, whence $Mg\subseteq X_n(G)$. 
\end{proof}

\section{\bf Compact groups with splitting automorphisms of order $3$} \label{large}

In this section we prove $c_3<1$.

\begin{thm}\label{3split}
	Let $G$ be a compact group and $\alpha$ be an automorphism of $G$ such that $\alpha^3={\rm id}$ and the set $X=\{x\in G \;|\; x x^\alpha x^{\alpha^2}=1\}$ is measurable with ${\mathbf m}(X)>\frac{15}{16}$. Then $X=G$.
\end{thm}
\begin{proof}
First we prove that $G$ is $2$-Engel. The proof is similar to an argument used in the proof of \cite[Theorem 4.4]{SAE}. We give the proof for the reader's convenience.
For any $a,b\in G$ we must prove that $[a,b,b]=1$.  Consider the set 
$$M:=X \cap b^{-1} X \cap aX \cap a^{-1}X \cap ab^{-1}X \cap ba^{-1}X \cap abX \cap b^{-1} a^{-1}X.$$
 Since ${\mathbf m}(M)> \frac{1}{2}$, there exists $x\in X$ such that 
 \begin{align*}
 1&=(x\alpha^{-1})^3=(bx\alpha^{-1})^3=(ax\alpha^{-1})^3=(a^{-1}x\alpha^{-1})^3=(ab^{-1}x\alpha^{-1})^3\\
 &=(ba^{-1}x\alpha^{-1})^3=(abx\alpha^{-1})^3= (b^{-1}a^{-1}x\alpha^{-1})^3,
 \end{align*}
 where we are writing in the semidirect product $G\rtimes \langle \alpha \rangle$ by noting that 
 $g\in X$ if and only if $(g\alpha^{-1})^3=1$ in $G\rtimes \langle \alpha \rangle$.  Now \cite[Lemma 4.1]{SAE} implies that $[a,b,b]=1$.\\

Let $X^{-1}=\{x^{-1} \;|\; x\in X\}$ and $Y=X \cap X^{-1}$.  Then ${\mathbf m}(Y)>\frac{7}{8}$. Note that for all $x\in Y$, 
$$\langle x,x^\alpha,x^{\alpha^2}\rangle=\langle x,x^\alpha \rangle=\langle x^\alpha,x^{\alpha^2}\rangle=\langle x,x^{\alpha^2}\rangle \;\; {\rm are \; all \; abelian.} \;\;\;(*)$$ By \cite[Theorem 7.15 (iv)]{R}, $$[a,b,c]=[a,c,b]^{-1} \;\; {\rm for \; all} \; a,b,c\in G. \;\; (**)$$ 
It follows from $(*)$ and $(**)$ that 
$$[a,b,c]=1 \; {\rm for \; all \;} a,b,c\in\{x,x^\alpha,x^{\alpha^2},h\}, \;\; {\rm for \; all\;} h \in G. \;\; (1)$$

Let $g$ be an arbitrary element of $G$. Consider the set $Z=Y \cap g^{-1} Y$. Since ${\mathbf m}(Z)>\frac{3}{4}$, $Z\neq \varnothing $ and for all  $x\in Z$ we have
$$xx^\alpha x^{\alpha^2}=gx (gx)^\alpha (gx)^{\alpha^2}=1. \;\;\;(2)$$ 
Using $(1)$ and $(2)$ and since $G$ is nilpotent of class at most $3$ (see e.g. \cite[Corollary 3 page 45]{R}), it follows that 
$$(gg^\alpha g^{\alpha^2})^{-1}=[x,g^\alpha] [x,g^{\alpha},g^{\alpha^2}] [g^{\alpha^2},x^{\alpha^2}] \;\; {\rm for \; all\;} x \in Z \;\; (3)$$
Now consider $W=Z \cap x_0^{-1}Z$ for some $x_0\in Z$. Since ${\mathbf m}(W)>\frac{1}{2}$, $W$ is nonempty and for all  $y \in W$ we have 
 $$(gg^\alpha g^{\alpha^2})^{-1}=[x_0y,g^\alpha] [x_0y,g^{\alpha},g^{\alpha^2}] [g^{\alpha^2},(x_0y)^{\alpha^2}]=[y,g^\alpha] [y,g^{\alpha},g^{\alpha^2}] [g^{\alpha^2},y^{\alpha^2}].$$
Note that since $G$ is nilpotent of class at most $3$, commutators of weight $3$ are central. It follows from $(3)$ that 
 $$(gg^\alpha g^{\alpha^2})^{-1}=[x_0,g^\alpha,y][g^{\alpha^2},x_0^{\alpha^2},y^{\alpha^2}] \;\; {\rm for \; all\;} y\in W. \;\; (4)$$
Now consider $T=W \cap y_0^{-1} W $ for some $y_0\in W$. Since ${\mathbf m}(T)>0$, there exists $z\in W$ such that $y_0z\in W$. It follows from $(4)$ that 
$$(gg^\alpha g^{\alpha^2})^{-1}=[x_0,g^\alpha,y_0z][g^{\alpha^2},x_0^{\alpha^2},(y_0z)^{\alpha^2}]=[x_0,g^\alpha,y_0][g^{\alpha^2},x_0^{\alpha^2},y_0^{\alpha^2}].$$
Since commutators of weight $3$ are central in $G$, 
$$[x_0,g^\alpha,y_0z][g^{\alpha^2},x_0^{\alpha^2},(y_0z)^{\alpha^2}]=[x_0,g^\alpha,y_0][g^{\alpha^2},x_0^{\alpha^2},y_0^{\alpha^2}] [x_0,g^\alpha,z][g^{\alpha^2},x_0^{\alpha^2},z^{\alpha^2}].$$
Therefore, $$[x_0,g^\alpha,z][g^{\alpha^2},x_0^{\alpha^2},z^{\alpha^2}]=1$$
and so, as $z\in W$, it follows from $(4)$ that $gg^\alpha g^{\alpha^2}=1$. This completes the proof.
\end{proof}

\begin{rem}\label{rm1}
	Using the first part of our proof of Theorem \ref{3split} and applying \cite[Theorem 6.5]{jw2020} (cf \cite[Th\'eor\`eme 5]{jw2000}) the number $\frac{15}{16}$ can be reduced to $\frac{7}{8}$. 
\end{rem}

\begin{rem}
	In \cite{MK} groups $G$ having an automorphism $\phi \in {\rm Aut}(G)$ with $\phi^n=1$ such that $X_{n,\phi}(G)$ is a large
 set in the sense of \cite{jw2000} for $n=3,4$ are studied. For the case $n=3$ it is proved in \cite{MK} that $G=X_{3,\phi}(G)$. For the case $n=4$, it is proved that a normal solvable subgroup $H$ of $G$ of derived length $d\geq 3$ is nilpotent of class at most $\frac{9^{d-2}+1}{2}$. It is interesting to know if the same latter result is valid  we replace ``largness" of $X_{n,\phi}(G)$ by  weaker condition ``$k$-largness" for some $k$ in the sense of \cite{jw2020}. If the latter is valid, then for compact groups $G$ with ${\mathbf m}_G(X_{4,\phi}(G))>1-\frac{1}{k}$ and $\phi$ is continuous   by Lemma \ref{n-large} we have the same results.  
\end{rem}

\begin{thm}\label{c3<7/8}
	$c_3 < 1$.
\end{thm}
\begin{proof}
It follows from Theorem \ref{3split}.	
\end{proof}

We now see that  Conjecture \ref{l-pcon} of  L\'evai  and  Pyber is true for $n=3$.

\begin{thm}
	Let $G$ be a profinite group such that the equation $x^3=1$ holds on a set with positive Haar measure. 
	Then the solutions set of the equation $x^3=1$ has non-empty interior.
\end{thm}
\begin{proof}
It follows from Theorem \ref{c3<7/8} and Proposition \ref{cj <- cj}, 	
\end{proof}

\end{document}